\documentclass[11pt]{amsart}
\usepackage{amsfonts}
\usepackage{}
\usepackage{mathrsfs}
\usepackage{bbm}
\usepackage{amssymb}
\usepackage{indentfirst}
\usepackage{latexsym,amsfonts,amssymb,amsmath}
\newtheorem{theorem}{Theorem}[section]

\newtheorem{corollary}[theorem]{Corollary}

\theoremstyle{definition}
\newtheorem{definition}[theorem]{Definition}
\newtheorem{proposition}[theorem]{Proposition}
\theoremstyle{remark}

\newtheorem{example}[theorem]{Example}

\begin{document}

\title
{Fuzzy gyronorms on gyrogroups}
\author{Li-Hong Xie}
\address{(L.H. Xie) School of Mathematics and Computational Science, Wuyi University, Jiangmen 529020, P.R. China} \email{xielihong2011@aliyun.com;yunli198282@126.com}


\thanks{This work is supported by Natural Science Foundation of China (Grant No.11871379, 11526158) and the Innovation Project of Department of Education of Guangdong Province, China.}



\begin{abstract}
 The concept of gyrogroups is a generalization of groups which
do not explicitly have associativity. In this paper, the notion of fuzzy gyronorms on gyrogroups is introduced. The relations of fuzzy metrics (in the sense of George and Veeramani), fuzzy gyronorms and gyronorms on gyrogroups are studied. Also, the fuzzy metric structures on fuzzy normed gyrogroups are discussed. In the last, the fuzzy metric completion of a gyrogroup with an invariant metric are studied. We mainly show that let $d$ be an invariant metric on a gyrogroup $G$ and $(\widehat{G},\widehat{d})$ is the metric completion of the metric space $(G,d)$; then for any continuous $t$-norm $\ast$, the standard fuzzy metric space $(\widehat{G},M_{\widehat{d}},\ast)$ of $(\widehat{G},\widehat{d})$ is the (up to isometry) unique fuzzy metric completion of the standard fuzzy metric space $(G,M_d,\ast)$ of $(G,d)$; furthermore, $(\widehat{G},M_{\widehat{d}},\ast)$ is a fuzzy metric gyrogroup containing $(G,M_d,\ast)$ as a dense fuzzy metric subgyrogroup and $M_{\widehat{d}}$ is invariant on $\widehat{G}$. Applying this result, we obtain that every gyrogroup $G$ with an invariant metric $d$ admits an (up to isometric) unique complete metric space $(\widehat{G},\widehat{d})$ of $(G,d)$ such that $\widehat{G}$ with the topology introduced by $\widehat{d}$ is a topology gyrogroup containing $G$ as a dense subgyrogroup and $\widehat{d}$ is invariant on $\widehat{G}$.
\end{abstract}

\maketitle

\section{Introduction}
 Taking as a point of starting the notion of a Menger space, Kramosil and Michalek introduced a notion of metric fuzziness \cite{KM} which became an interesting and fruitful area of research(see for example\cite{GSa,MSa,RaB,RiR}). Furthermore, fuzzy metric spaces have been investigated by several
authors from different points of view (see for example \cite{De,Erc,KS}). In particular, George and Veeramani \cite{GeVe},
by modifying a definition of fuzzy metric space given
by Kramosil and Michalek \cite{KM}, have introduced and studied a new and interesting notion of a fuzzy metric space with the help of
continuous $t$-norms.

Recall that a binary operation $\ast: [0, 1] \times [0, 1] \rightarrow [0, 1]$ is a {\it continuous $t$-norm} \cite{SS}
if $\ast$ satisfies the following conditions:
\begin{enumerate}
\item [(i)] $\ast$ is associative and commutative;
\item [(ii)] $\ast$ is continuous;
\item [(iii)] $a\ast 1=a$ for all $a\in [0,1]$;
\item [(iv)] $a\ast b\leq c\ast d$ whenever $a \leq c$ and $b \leq d$, with $a, b, c, d \in [0, 1]$.
\end{enumerate}

Three paradigmatic examples of continuous $t$-norms are $\wedge$, $\cdot$ and $\ast_L$ (the Lukasiewicz $t$-norm), which are defined by
$a\wedge b = \text{min} \{ a, b \}$, $a \cdot b = ab$ and $a \ast_L b = \text{max} \{ a + b-1, 0 \}$, respectively. One can easily show that $\ast\leq \wedge$ for every continuous $t$-norm $\ast$.


\begin{definition}(in the sense of George and Veeramani \cite{GeVe})\label{Def:M}
 A {\it fuzzy metric}  on a set $X$ is a pair $(M, \ast)$ such
 that $M$ is a fuzzy set in $X \times X \times (0,  +\infty)$ and $\ast$ is a continuous $t$-norm satisfying for all $x, y, z \in X$ and $t,s>0$:
 \begin{enumerate}
\item [(i)] $M(x,y,t)>0$;
\item [(ii)]$M(x,y,t)=1$ if and only if $x=y$;
\item [(iii)] $M(x, y, t) = M(y, x, t)$;
\item [(iv)] $M(x,y,t+s)\geq M(x,z,t)\ast M(z,y,s)$;
\item [(v)] $M(x,y,_-):(0, + \infty)\rightarrow [0,1]$ is continuous.\\
\end{enumerate}
\end{definition}

By a {\it fuzzy metric space} we mean an ordered triple $(X, M, \ast)$ such that $X$ is a set and $(M, \ast)$ is a fuzzy metric on $X$. It is known that every fuzzy metric $(M, \ast)$ on a set $X$ induces a topology $\tau_M$ on $X$, which has as a base the family of open sets of the form $\{B_M(x,\varepsilon,t):x\in X, \varepsilon\in (0,1), t>0\}$, where $B_M(x,\varepsilon,t)=\{y\in X:M(x,y, t)>1-\varepsilon\}$ for all $x\in X$, $\varepsilon\in (0,1)$, $t>0$ (see \cite{GeVe}).

Combinations of a fuzzy metric structure and an algebraic
structure deserve special attention in fuzzy Topological Algebra. The most frequently studied structures fall into the
so-called fuzzy normed spaces (among others, the interested reader can consult \cite{AlR, BaS, ChM, FeF, Lee, AnA}). Also, fuzzy metrics on groups are studied by several scholars (see,\cite{GuR,RoS,Sa2,Sa}). They find some sufficient conditions to make some topological algebraic structures (in particular a nonsymmetric structure) become stronger topological structures (in particular, a symmetric structure). In particular, recently, S\'{a}nchez and Sanchis proved that the completion of a fuzzy metric group $(G,M,\ast)$ such that $(M,\ast)$ is invariant on $G$ is a fuzzy metric group (in the sense of Kramosil and Michalek)\cite[Theorem 2.2]{Sa2}.

In \cite{Ung}, Ungar studies a parametrization of the Lorentz transformation group.
This leads to the formation of gyrogroup theory, a rich subject in mathematics (among others, the interested reader can consult \cite{AbW,Ferr}). Loosely speaking, a gyrogroup (see Definition \ref{Def:gyr}) is a group-like structure in which the associative law fails to satisfy. Recently, topological gyrogroups are studied by Atiponrat \cite{Atip} and Cai et al \cite{Cai} and so on. In particular, Cai et al \cite{Cai} extended the famous Birkhoff-Kakutani theorem by proving that every first-countable Hausdorff topological gyrogroup is metrizable \cite[Theorem 2.3]{Cai}. Recently, Suksumran \cite{Suk} studied the normed gyrogroup. In particular, Suksumran proved that the normed gyrogroups are homogeneous and form left invariant
metric spaces and derive a version of the Mazur-Ulam theorem. Also, Suksumran given
certain sufficient conditions, involving the right-gyrotranslation inequality
and Klee's condition, for a normed gyrogroup to be a topological gyrogroup (see \cite{Suk}).

Those lead to the notion of a fuzzy normed gyrogroup is introduced in this paper. We mainly study the fuzzy metrics structures and the fuzzy metrics completion on fuzzy normed gyrogroups. The paper is organized as follows. In Section \ref{Sec:basic}, some basic facts and definitions are stated. Section \ref{Sec:Fuzzy} is devoted to study the fuzzy normed gyrogroups. The relations of fuzzy gyronorms, fuzzy metrics and gyronorms on gyrogroups are studied. We mainly show that: Every fuzzy normed gyrogroup $G$ has an invariant fuzzy metric under left gyrotranslations on $G$ and every gyrogroup $G$ with an invariant fuzzy metric under left gyrotranslations is a fuzzy normed gyrogroup (see Theorems \ref{The:M}, \ref{THe} and \ref{THM}). Also, some sufficient conditions, which make a fuzzy normed gyrogroup to be a topological gyrogroup, are found (see Theorem \ref{Them:F}). In Section \ref{Sec:Com} we consider the fuzzy metric completion of an invariant metric gyrogroup by proving that let $d$ be an invariant metric on a gyrogroup $G$ and $(\widehat{G},\widehat{d})$ is the metric completion of the metric space $(G,d)$; then for any continuous $t$-norm $\ast$, the standard fuzzy metric space $(\widehat{G},M_{\widehat{d}},\ast)$ of $(\widehat{G},\widehat{d})$ is the (up to isometry) unique fuzzy metric completion of the standard fuzzy metric space $(G,M_d,\ast)$ of $(G,d)$; furthermore, $(\widehat{G},M_{\widehat{d}},\ast)$ is a fuzzy metric gyrogroup containing $(G,M_d,\ast)$ as a dense fuzzy metric subgyrogroup and $M_{\widehat{d}}$ is invariant on $\widehat{G}$ (see Theorem \ref{The:comp}).
 Applying this result, we obtain that every gyrogroup $G$ with an invariant metric $d$ admits an (up to isometric) unique complete metric space $(\widehat{G},\widehat{d})$ of $(G,d)$ such that $\widehat{G}$ with the topology introduced by $\widehat{d}$ is a topology gyrogroup containing $G$ as a dense subgyrogroup and $\widehat{d}$ is invariant on $\widehat{G}$ (see Corollary \ref{CC}).

\section{Basic facts and definitions}\label{Sec:basic}
The concept of gyrogroups as a generalization of groups, is originated from the study of $c$-ball of relativistically admissible velocities with Einstein
velocity addition as mentioned by Ungar in \cite{Ung}.

Let $G$ be a nonempty set, and let $\oplus  : G  \times G \rightarrow G $ be a binary operation on $G $. Then the pair $(G, \oplus)$ is
called a {\it groupoid.}  A function $f$ from a groupoid $(G_1, \oplus_1)$ to a groupoid $(G_2, \oplus_2)$ is said to be
a groupoid homomorphism if $f(x_1\oplus_1 x_2)=f(x_1)\oplus_2 f(x_2)$ for any elements $x_1, x_2 \in G_1$.  In addition, a bijective
groupoid homomorphism from a groupoid $(G, \oplus)$ to itself will be called a groupoid automorphism. We will write $\text{Aut~} (G, \oplus)$ for the set of all automorphisms of a groupoid $(G, \oplus)$.
\begin{definition}\cite[Definition 2.7]{Ung}\label{Def:gyr}
 Let $(G, \oplus)$ be a nonempty groupoid. We say that $(G, \oplus)$ or just $G$
(when it is clear from the context) is a gyrogroup if the followings hold:
\begin{enumerate}
\item[($G1$)] There is an identity element $e \in G$ such that
$$e\oplus x=x \text{~~~~~for all~~}x\in G.$$
\item[($G2$)] For each $x \in G $, there exists an {\it inverse element}  $\ominus x \in G$ such that
$$\ominus x\oplus x=e.$$
\item[($G3$)] For any $x, y \in G $, there exists an {\it gyroautomorphism} $\text{gyr}[ x, y ] \in \text{Aut}( G, \oplus)$ such that
$$x\oplus (y\oplus z)=(x\oplus y)\oplus \text{gyr}[ x, y ](z)$$ for all $z \in G$.
\item[($G4$)] For any $x, y \in G$, $\text{gyr}[ x \oplus y, y ] = \text{gyr}[ x, y ]$.
\end{enumerate}
\end{definition}

One can easily show that any gyrogroup has a unique two-sided identity $e$, and an element $a$ of the
gyrogroup has a unique two-sided inverse $\ominus a$. It is clear that every group satisfies
the gyrogroup axioms (the gyroautomorphisms are the identity map) and hence is
a gyrogroup. Conversely, any gyrogroup with trivial gyroautomorphisms forms a
group. From this point of view, gyrogroups naturally generalize groups.

Proposition \ref{Pro:gyr} summarizes some algebraic properties of gyrogroups, which will prove useful in studying topological and geometric aspects of gyrogroups in Sections \ref{Sec:Fuzzy} and \ref{Sec:Com}.
\begin{proposition}(\cite{Su1,Su})\label{Pro:gyr}
Let $(G,\oplus)$ be a gyrogroup and $a,b,c\in G$. Then
\begin{enumerate}
\item[(1)] $\ominus(\ominus a)=a$ \hfill{Involution of inversion}
\item[(2)] $\ominus a\oplus(a\oplus b)=b$ \hfill{Left cancellation law}
\item[(3)] \text{gyr}$[a,b](c)=\ominus(a\oplus b)\oplus(a\oplus(b\oplus c))$ \hfill{Gyrator identity}
\item[(4)] $\ominus(a\oplus b)=\text{gyr}[a,b](\ominus b\ominus a)$\hfill{\text{cf.~}$(ab)^{-1}=b^{-1}a^{-1}$}
\item[(5)] $(\ominus a\oplus b)\oplus \text{gyr}[\ominus a,b](\ominus b\oplus c)=\ominus a\oplus c$ \hfill{\text{cf.~}$(a^{-1}b)(b^{-1}c)=a^{-1}c$}
\item[(6)] $\text{gyr}[a,b]=\text{gyr}[\ominus b,\ominus a]$ \hfill{Even property}
\item[(7)] $\text{gyr}[a,b]=\text{gyr}^{-1}[b,a], \text{the inverse of gyr}[b,a]$ \hfill{Inversive symmetry}
\end{enumerate}
\end{proposition}

As far as we known, Atiponrat is the first scholar who extended the idea of topological groups to topological gyrogroups as gyrogroups with a topology such that
its binary operation is jointly continuous and the operation of taking the inverse is continuous.

\begin{definition}\cite[Definition 1]{Atip}
A triple $( G, \tau,  \oplus)$ is called a {\it topological gyrogroup} if and only if
\begin{enumerate}
\item[(1)] $( G, \tau)$ is a topological space;
\item[(2)] $( G, \oplus)$ is a gyrogroup; and
\item[(3)] The binary operation  $\oplus: G  \times G  \rightarrow G$ is continuous where $G \times G$ is endowed with the product topology
and the operation of taking the inverse $\ominus(\cdot ) : G  \rightarrow G $, i.e. $x \rightarrow \ominus x$, is continuous.
\end{enumerate}
\end{definition}

If a triple $( G, \tau,  \oplus)$ satisfies the first two conditions and its binary operation is continuous, we call such
triple a {\it paratopological gyrogroup} \cite{Atip1}. Sometimes we will just say that $G$ is a topological gyrogroup (paratopological gyrogroup) if the binary operation and the topology are clear from the context.

Clearly, every topological group is a topological gyrogroup. Hence, tt is natural to ask for the existence of a topological gyrogroup which is not a topological gyrogroup. In fact, Atiponrat show that the M\"{o}bius gyrogroup and the Einstein gyrogroups with the standard topology are such examples \cite[Examples 2 and 3]{Atip}. For the sake of completeness, we give one of Examples as follows.

\begin{example}\cite[Example 2]{Atip}
 Let $\mathbb{D}$ be the complex open unit disk $\{z\in \mathbb{C}:|z|<1\}$. Consider $\mathbb{D}$ with the standard topology. Next, we define a M\"{o}bius addition $\oplus_{\text{M}}:\mathbb{D}\times \mathbb{D}\rightarrow \mathbb{D}$ to be a function such that
 $$a\oplus_{\text{M}} b=\frac{a+b}{1+\overline{a}b}\quad\quad\quad\quad\quad\text{for all~}a,b\in\mathbb{D}.$$
 Then $\mathbb{D}$ with the operator $\oplus_{\text{M}}$ is not a group, which has no associativity. However, it has been proved in section 3.4 of \cite{Ung} that $(\mathbb{D}, \oplus_{\text{M}})$ is a gyrogroup where the gyroautomorphism define as follows: for any $a,b,c\in \mathbb{D}$
 $$\text{gyr}[a,b](c)=\frac{1+a\overline{b}}{1+\overline{a}b}c.$$

 This gyrogroup is one of the most important examples of gyrogroups. It is called the {\it M\"{o}bius gyrogroup}.
 Moreover, $0$ is the identity, and for any $a\in \mathbb{D}$, we get that $-a\in \mathbb{D}$ such that $-a \oplus_{\text{M}} a=0$. Furthermore, $\mathbb{D}$ with the standard topology, the operator $\oplus_{\text{M}}$ and the inverse operator are continuous, so $(\mathbb{D}, \oplus_{\text{M}})$ is a topology gyrogroup, but not a topological group or or a paratopological group.
\end{example}

\section{Fuzzy normed gyrogroups}\label{Sec:Fuzzy}

In this section we shall introduced the notions of fuzzy gyronorm on gyrogroups. Also, the fuzzy metric structure and geometric structures of fuzzy normed gyrogroups are studied. Let us begin with the following definition.

\begin{definition}\cite[Definition 3.1]{Suk}\label{Def:No}
({\bf Gyronorms}). Let $(G,\oplus)$ be a gyrogroup. A function $\parallel\cdot\parallel: G\rightarrow\mathbb{R}$ is
called a gyronorm on $G$ if the following properties hold:
\begin{enumerate}
\item $\parallel x\parallel\geq 0$ \text{~for each~}$x\in G$ \text{~and~}$\parallel x\parallel= 0$ \text{~ if and only if~} $x=e$;\hfill{\text{(positivity)}}
\item $\parallel \ominus x\parallel= \parallel x\parallel$ \text{~for each~}$x\in G$; \hfill{\text{(invariant under taking inverses)}}
\item $\parallel x\oplus y\parallel\leq\parallel x\parallel+\parallel y\parallel$; \hfill{\text{~for each~}$x,y\in G $ \text{(subadditivity)}}
\item $\parallel\text{gyr}[a,b](x)\parallel=\parallel x\parallel$ \text{~for each~}$x,a,b\in G$. \hfill{\text{(invariant under gyrations)}}
\end{enumerate}
\end{definition}

In 2018, Suksumran \cite{Suk} introduced the notions of gyronorms on gyrogroups and said that any gyrogroup with a gyronorm is called a {\it normed gyrogroup}. Also, some intersting metric and geometric structures of normed gyrogroups are established in \cite{Suk}. This leads us to introduce the notions of Fuzzy Gyronorms on gyrogroup as follows.

\begin{definition} ({\bf Fuzzy Gyronorms})\label{Def:N}
  Given a gyrogroup $(G,\oplus)$ and a continuous $t$-norm $\ast$, the pair $(N,\ast)$ is called a {\it fuzzy gyronorms} on $G$ if $N$ is a fuzzy set of $G\times (0,+\infty)$ satisfying the following conditions for all $x, y \in G$ and all $t,s\in (0,+\infty)$:
 \begin{enumerate}
 \item[($N1$)] $N(x,t)>0$;
 \item[($N2$)] $x=e$ if and only if $N(x,t)=1$;
 \item[($N3$)] $N(\ominus x,t)=N(x,t)$;
 \item[($N4$)] $N(x\oplus y,t+s)\geq N(x,t)\ast N(y,s);$
 \item[($N5$)] $N(x,_{-}):(0,+\infty)\rightarrow [0,1]$ is continuous;
 \item[($N6$)] $N(\text{gyr}[a,b])(x),t)=N(x,t)$ for all $a,b\in G$.
 \end{enumerate}
\end{definition}

 Given a gyrogroup $G$, a continuous $t$-norm $\ast$  and a fuzzy set $N$ of $G\times (0,+\infty)$, the ordered triple $(G,N,\ast)$ is called a {\it fuzzy normed} gyrogroup if $(N,\ast)$ is a fuzzy gyronorm on $G$.

 The relations of normed gyrogroups and fuzzy normed gyrogroups are shown as follows.

\begin{proposition}\label{PO}
 Let $(G,\parallel\cdot\parallel)$ be a normed gyrogroup and define a fuzzy set $N_{\parallel\cdot\parallel}$ of $G\times (0,+\infty)$ by $N_{\parallel\cdot\parallel}(x,t)=\frac{t}{t+\parallel x \parallel}$ for each $x\in G$ and $t>0$, then the ordered triple $(G,N_{\parallel\cdot\parallel},\ast)$ is a fuzzy normed gyrogroup, where $\ast$ is a contnuous $t$-norm;
\end{proposition}

\begin{proof}
Now let us check $(N_{\parallel\cdot\parallel},\ast)$ satisfies the conditions in Definition \ref{Def:N}. In fact, it is enough to show that $N_{\parallel\cdot\parallel}$ satisfies Definition \ref{Def:N} ($N4$), since clearly other conditions hold for $N_{\parallel\cdot\parallel}$ by Definition \ref{Def:No}.

In fact, take any $x,y\in G$ and $t,s>0$. Then $$N_{\parallel\cdot\parallel}(x\oplus y,t+s)=\frac{s+t}{s+t+\parallel x\oplus y\parallel}\geq \frac{s+t}{s+t+\parallel x\parallel + \parallel y\parallel},$$
since $\parallel x\oplus y\parallel \leq \parallel x\parallel + \parallel y\parallel$ by Definition \ref{Def:No} (3). Thus, to show that
$$N_{\parallel\cdot\parallel}(x\oplus y,t+s)\geq N_{\parallel\cdot\parallel}(x,t)\ast N_{\parallel\cdot\parallel}(y,s)=\frac{t}{t+\parallel x \parallel}\ast \frac{s}{s+\parallel y \parallel},$$
it is enough to show that $$\frac{s+t}{s+t+\parallel x\parallel + \parallel y\parallel}\geq \text{min~}\{\frac{t}{t+\parallel x \parallel},\frac{s}{s+\parallel y\parallel}\},$$
since $\ast$ is a continuous $t$-norm.

Without loss of generality, we assume that $$\frac{t}{t+\parallel x \parallel}\geq\frac{s}{s+\parallel y \parallel},$$
which is equivalent to $$\frac{\parallel x\parallel}{t}\leq \frac{\parallel y\parallel}{s}.$$
Hence, we have that
\begin{align*}
&\frac{s+t}{s+t+\parallel x\parallel + \parallel y\parallel}\geq \text{min~}\{\frac{t}{t+\parallel x \parallel},\frac{s}{s+\parallel y \parallel}\}
\\&\Leftrightarrow \frac{s+t}{s+t+\parallel x\parallel + \parallel y\parallel}\geq \frac{s}{s+\parallel y \parallel}
\\&\Leftrightarrow \frac{s+t+\parallel x\parallel + \parallel y\parallel}{s+t}\leq \frac{s+\parallel y \parallel}{s}
\\&\Leftrightarrow \frac{\parallel x\parallel + \parallel y\parallel}{s+t}\leq \frac{\parallel y \parallel}{s}
\\&\Leftrightarrow s\parallel x\parallel+s\parallel y\parallel\leq s\parallel y\parallel+t\parallel y\parallel
\\&\Leftrightarrow  s\parallel x\parallel\leq t\parallel y\parallel
\\&\Leftrightarrow \frac{\parallel x\parallel}{t}\leq \frac{\parallel y\parallel}{s}.
\end{align*}
Hence, by the assumption, we have proved that $$N_{\parallel\cdot\parallel}(x\oplus y,t+s)\geq N_{\parallel\cdot\parallel}(x,t)\ast N_{\parallel\cdot\parallel}(y,s)$$
holds for each $x,y\in G$ and $s,t>0$.
\end{proof}

Proposition \ref{PO} shows that every normed gyrogroup can introduce a fuzzy normed gyrogroup, for example, on the Einstein gyrogroups and the M\"{o}biusg yrogroups there are fuzzy gyronorms, since they are normed gyrogroups (see \cite[Theorems 3.3 and 3.4]{Suk}).
Theorems \ref{The:M} and \ref{THM} reveals the relations of fuzzy metrics and fuzzy gyronorms on gyrogroups.
\begin{theorem}\label{The:M}
 Let $(G,N,\ast)$ be a fuzzy normed gyrogroup. Define
 $$M_N(x,y,t)=N(\ominus x\oplus y,t)$$
 for all $x,y\in G$ and $t>0$. Then $(M,\ast)$ is a fuzzy metric on $G$.
\end{theorem}

\begin{proof}
Suppose that $(G,N,\ast)$ is a fuzzy normed gyrogroup. Let us verify that the pair $(M_N,\ast)$ satisfies (i)-(v) in Definition \ref{Def:M}.

(i): According to the definition of $M_N$, $M_N(x,y,t)=N(\ominus x\oplus y,t)$, so $M_N(x,y,t)>0$ by Definition \ref{Def:N}(N1) $N(\ominus x\oplus y,t)>0$.

(ii): If $x=y$, then $M_N(x,y,t)=N(\ominus x\oplus y,t)=N(e,t)=1$ for all $t>0$ by Definition \ref{Def:N}$(N2)$. If $M_N(x,y,t)=1$  for all $t>0$, then $N(\ominus x\oplus y,t)=1$ for all $t>0$. Thus, by Definition \ref{Def:N}$(N2)$ we obtain that $\ominus x\oplus y=e$, so $x=y$ by by the left cancellation law.

(iii): By  Proposition \ref{Pro:gyr} (4) and Definition \ref{Def:N} ($N3$) and ($N6$) we obtain that
\begin{align*}
M_N(x,y,t)&=N(\ominus x\oplus y,t)
\\&=N(\ominus(\ominus x\oplus y),t)
\\&=N(\text{gyr}[\ominus x,y](\ominus y\oplus x),t)
\\&=N(\ominus y\oplus x),t)
\\&=M_N(y,x, t).
\end{align*}

(iv): Take any $x,y,z\in G$. Then Proposition \ref{Pro:gyr} (5) and Definition \ref{Def:N} ($N4$), ($N6$) we obtain that

\begin{align*}
M_N(x,z,s+t)&=N(\ominus x\oplus z,s+t)
\\&=N((\ominus x\oplus y)\oplus \text{gyr}[\ominus x,y](\ominus y\oplus z),s+t)
\\&\geq N(\ominus x\oplus y,s)\ast N(\text{gyr}[\ominus x,y](\ominus y\oplus z),t)
\\&=N(\ominus x\oplus y,s)\ast N(\ominus y\oplus z,t)
\\&=M_N(x,y,s)\ast M(y,z,t).
\end{align*}

(v): It is obvious by Definition \ref{Def:N} ($N5$).
\end{proof}

 The fuzzy metric induced by a fuzzy gyronorm in Theorem \ref{The:M} is called a {\it fuzzy gyronorm metric} on $G$.

 Let $(X, M, \ast )$ and $(Y, N, \star )$ be two fuzzy metric spaces. Recall that a mapping from $X$ to $Y$  is called
an {\it isometry} if for each $x, y \in X$ and each $t>0$, $M(x,y,t)=N(f(x),f(y),t)$. Two fuzzy metric spaces $(X, M, \ast )$ and $(Y, N, \star )$ are called {\it isometric} if there is an isometry from $X$ onto $Y$ (see \cite[Definitions 1 and 2]{GrSR}).

\begin{theorem}\label{THe}
Let $(G,N,\ast)$ be a fuzzy normed gyrogroup. Then the fuzzy gyronorm metric $(M_N,\ast)$ with respect to $(N,\ast)$ is invariant under left gyrotranslation:
$$M_N(a\oplus x,a \oplus y,t)=M_N(x,y,t)$$
for all $x,y,a\in G$ and $t\in (0,+\infty)$.  Hence, every left gyrotranslation of $G$ is an isometry of $(G,M_N,\ast)$.
\end{theorem}

\begin{proof}
Let $a,x,y\in G$. Recall that the {\it left gyrotranslation} by $a$, denote by $L_a$, is defined by $L_a(x)=a\oplus x$ for all $x\in G$.
Next, we prove that the fuzzy gyronorm metric $(M_N,\ast)$ is invariant under $L_a$. In fact,
\begin{align*}
M_N(L_a(x),L_a( y),t)&=M_N(a\oplus x,a \oplus y,t)
\\&=N(\ominus(a\oplus x)\oplus(a \oplus y),t)
\\&=N(\text{gyr}[a,x](\ominus x\ominus a)\oplus (a \oplus y),t) \quad \quad\quad\quad\text{by Proposition \ref{Pro:gyr} (4)~}
\\&=N((\ominus x\ominus a)\oplus \text{gyr}[x,a](a \oplus y),t) \quad\quad\quad\quad~~\text{by Proposition }\ref{Pro:gyr} (7)
\\&\quad\quad\quad\quad\quad\quad\quad\quad\quad\quad\quad\quad\quad\quad\quad\quad\quad\quad\quad\quad\text{and Definition \ref{Def:N}~} (N6)
\\&=N(\ominus x\oplus y,t)\quad \quad\quad\quad\quad\quad\quad\quad\quad\quad\quad\quad\quad\text{by Proposition \ref{Pro:gyr} (5)}
\\&=M_N(x,y,t).
\end{align*}
\end{proof}

\begin{corollary}
If $(G,N,\ast)$ is a fuzzy normed gyrogroup, then $G$ with the topology introduced by the fuzzy gyronorm metric with respect to $(N,\ast)$ is a left topological gyrogroup (every left gyrotranslation of $G$ is continuous). Hence, $G$ is homogeneous.
\end{corollary}

\begin{proof}
The fact that $G$ with the topology introduced by the fuzzy gyronorm metric with respect is a left topological gyrogroup directly follows from Theorem \ref{THe}. Take any $x,y\in G$. Then it is obvious that $L_{y}(L_{\ominus x}(x))=y$, so by Theorem \ref{THe} $L_{y}\circ L_{\ominus x}$ is as required. Hence $G$ is homogeneous.
\end{proof}

\begin{theorem}\label{THM}
 Let $G$ be a gyrogroup with a fuzzy metric $(M,\ast)$. If $(M,\ast)$ is invariant under left
gyrotranslation, that is, $$M(a\oplus x,a\oplus y,t)=M(x,y,t)$$
for all $a,x,y \in G$ and $t>0$ and define $N_M(x,t)= M(e,x,t)$ then $(N_M,\ast)$ is a fuzzy gyronorm on $G$ that generates the
same fuzzy metric.
 \end{theorem}

\begin{proof}
Clearly, $N_M$ is a fuzzy set of $G\times(0,+\infty)$. Let us check $(N_M,\ast)$ satisfies the conditions in Definition \ref{Def:N}.

($N1$): $N_M(x,t)=M(e,x,t)>0$ holds for each $x\in G$ and $t>0$ by Definition \ref{Def:M}(i).

($N2$): By Definition \ref{Def:M}(ii), if $x=e$, then $N_M(e,t)=M(e,e,t)=1$ holds for each $t>0$. Conversely, if $N_M(x,t)=M(e,x,t)=1$ holds for each $t>0$, then $x=e$.

($N3$): Since $(M,\ast)$ is invariant under left
gyrotranslation, we have that the equalities
$$
N_M(\ominus x,t)=M(e,\ominus x,t)
=M(x,e,t)
=M(e,x,t)
=N_M(x,t)
$$
hold for each $t>0$.

($N4$): By Definition \ref{Def:M}(iii), (iv) and that the invariantness of $(M,\ast)$ under left
gyrotranslation, we have that
\begin{align*}
N_M(x\oplus y,t+s)&=M(e,x\oplus y,t+s)
\\&=M(\ominus x, y,t+s)
\\&\geq M(\ominus x, e,t)\ast M(e, y,s)
\\&=M( e, x,t)\ast M(e, y,s)
\\&=N_M(x,t)\ast N_M(y,s)
\end{align*}
hold for each $x,y\in G$ and $t,s>0$.

($N5$): It is obvious by Definition \ref{Def:M}(v).

($N6$): By the invariantness of $(M,\ast)$ under left
gyrotranslation, we have that

\begin{align*}
N_M(\text{gyr}[a,b](x),t)&=N_M(\ominus(a\oplus b)\oplus(a\oplus(b\oplus x),t)
\\&=M(e,\ominus(a\oplus b)\oplus(a\oplus(b\oplus x),t)
\\&=M(a\oplus b,a\oplus(b\oplus x),t)
\\&=M( b,b\oplus x),t)
\\&=M( e,x),t)
\\&=N_M(x,t)
\end{align*}
hold for each $a,b,x\in G$ and $t,s>0$.

Hence we have proved that $(N_M,\ast)$ is a fuzzy gyronorm on $G$.

Now according to Theorem \ref{THe}, since $(M,\ast)$ is invariant under left
gyrotranslation, we have that
$$
N_M(\ominus x\oplus y,t)=M(e,\ominus x\oplus y,t)=M(x,y,t)
$$
hold for each $x,y\in G$ and $t>0$. This implies that $(N_M,\ast)$ generates the
same fuzzy metric $(M,\ast)$.
\end{proof}

\begin{theorem}
Let $(G,N,\ast)$ be a fuzzy normed gyrogroup. If $\alpha\in \text{Aut~}G$ and $N(\alpha(x),t)=N(x,t)$ for all $x\in G$ and $t>0$, then $\alpha$ is an isometry of $(G,M_N,\ast)$, where $(M_N,\ast)$ is the fuzzy gyronorm metric with respect to $(N,\ast)$.
\end{theorem}

\begin{proof}
 By assumption, we have that
 \begin{align*}
 M_N(\alpha(x),\alpha(y),t)&=N(\ominus \alpha(x)\oplus \alpha(y),t)
 \\&=N(\ominus \alpha(x\oplus y),t)
 \\&=N(\ominus x\oplus y,t)
 \\&=M_N(x,y,t).
 \end{align*}
\end{proof}

\begin{corollary}
 If $(G,N,\ast)$ is a fuzzy normed gyrogroup, then the gyroautomorphisms of $G$ are
isometries of $(G,M_N,\ast)$, where $(M_N,\ast)$ is the fuzzy gyronorm metric with respect to $(N,\ast)$.
\end{corollary}

We have studied the fuzzy metric and geometric structures on fuzzy normed gyrogroups. Next, we shall study the topological structures on fuzzy normed gyrogroups. Some sufficient conditions which make gyrogroups with some topologies become topological gyrogroups are found.

\begin{theorem}\label{The:Klee}
 Let $(G,N,\ast)$ be a fuzzy normed gyrogroup and $(M_N,\ast)$ the fuzzy gyronorm metric with respect to $(N,\ast)$.
Consider the following conditions.
\begin{enumerate}
\item[$(I)$] Right-gyrotranslation inequality: $M_N(x\oplus a,y\oplus a,t)\geq M_N(x,y,t)$ for all $x,y,a\in G$ and $t>0$;
\item[$(I)'$] Klee's condition: $M_N(x\oplus y,a\oplus b,t+s)\geq M_N(x,a,t)\ast M_N(y,b,s)$ for all $x,y,a,b\in G$ and $t,s>0$;
\item[$(II)$]  Commutative-like condition: $N((a\oplus x)\oplus \text{gyr}[a,x](y\ominus a),t)=N( x\oplus y,t)$ for all $x,y,a\in G$ and $t>0$;
\item[$(II)'$]  Right-gyrotranslation invariant: $M_N(x\oplus a,y\oplus a,t)= M_N(x,y,t)$ for all $x,y,a\in G$ and $t>0$.
\end{enumerate}
Then $(II)\Leftrightarrow(II)'\Rightarrow (I)'\Leftrightarrow(I)$.
\end{theorem}

\begin{proof}
 The implication $(II)'\Rightarrow (I)'$ is obvious. Thus it is enough to show that $(II)\Leftrightarrow(II)'$ and $(I)'\Leftrightarrow(I)$.

 $(I)\Leftrightarrow(I)'$. Assume that the right-gyrotranslation inequality holds. Then we have that
 \begin{align*}
 M_N(x\oplus y,a\oplus b,t+s)&\geq M_N(x\oplus y,x\oplus b,s)\ast M_N(x\oplus b,a\oplus b,t)
 \\&= M_N(y,b,s)\ast M_N(x\oplus b,a\oplus b,t)
 \\&\geq M_N(x,a,t)\ast M_N(y,b,s).
 \end{align*}

 Conversely,
 \begin{align*}
 M_N(x\oplus a,y\oplus a,t)&\geq M_N(x,y,t)\ast M(a,a,0)
 \\&=M_N(x,y,t).
 \end{align*}

  $(II)\Leftrightarrow(II)'$. Assume commutative-like condition holds. Let $x,a,y\in G$ and  $t,s>0$. Then we have that
  \begin{align*}
  M_N(x\oplus a,y\oplus a,t)&=N(\ominus(x\oplus a)\oplus (y\oplus a),t)
  \\&=N(\text{gyr}[x,a](\ominus a\ominus x)\oplus (y\oplus a),t)
  \\&=N((\ominus a\ominus x)\oplus \text{gyr}[a,x](y\oplus a),t)
  \\&=N((\ominus a\ominus x)\oplus \text{gyr}[\ominus a,\ominus x](y\oplus a),t)
  \\&=N(\ominus x\oplus y,t)
  \\&=M_N(x,y,t).
  \end{align*}
  Conversely,
  \begin{align*}
  N(x\oplus y,t)&=M_N(\ominus x,y,t)
  \\&= M_N(\ominus x\ominus a,y\ominus a,t)
  \\&=N(\ominus(\ominus x\ominus a)\oplus(y\ominus a),t)
  \\&=N(\text{gyr}[\ominus x,\ominus a](a\oplus x)\oplus(y\ominus a),t)
  \\&=N((a\oplus x)\oplus\text{gyr}[\ominus a,\ominus x](y\ominus a),t)
  \\&=N((a\oplus x)\oplus\text{gyr}[a,x](y\ominus a),t).
  \end{align*}
\end{proof}

\begin{theorem}\label{Them:F}
 Let $(G,N,\ast)$ be a fuzzy normed gyrogroup. If one of the conditions $(I)$ and $(I)'$ in Theorem
\ref{The:Klee} holds, then $G$ is a topological gyrogroup endowed with the topology induced
by the fuzzy gyronorm metric $(M_N,\ast)$ with respect to $(N,\ast)$.
\end{theorem}

\begin{proof}
Firstly, we shall prove that the operator $\oplus$: $G\times G\rightarrow G$ is continuous. Take any $x,y\in G$
and any open neighborhood $V$ of $x \oplus y$. Then there are $\epsilon\in (0,1)$ and $t_0>0$ such that $B_{M_N}(x \oplus y, \epsilon,t_0)=\{z\in G:M_N(x\oplus y,z,t_0)>1-\epsilon \}\subseteq V$. Since the $\ast$ is a continuous $t$-norm, there is $\epsilon_0\in (0,1)$ such that $$(1-\epsilon_0,1]\ast (1-\epsilon_0,1]\subseteq (1-\epsilon, 1].$$ Then $B_{M_N}(x,\epsilon_0,\frac{t_0}{2})$ and $B_{M_N}(y,\epsilon_0,\frac{t_0}{2})$ are open neighborhoods of $x$ and $y$, respectively. Take any $a\in B_{M_N}(x,\epsilon_0,\frac{t_0}{2})$ and $b\in B_{M_N}(y,\epsilon_0,\frac{t_0}{2})$. Then by the conditions in Theorem
\ref{The:Klee} (II), $$M_N(x\oplus y,a\oplus b,t_0)\geq M_N(x,a,\frac{t_0}{2})\ast M_N(y,b,\frac{t_0}{2})>1-\epsilon,$$
since $M_N(x,a,\frac{t_0}{2})>1-\epsilon_0$ and $M_N(y,b,\frac{t_0}{2})>1-\epsilon_0$. This implies that $$B_{M_N}(x,\epsilon_0,\frac{t_0}{2})\oplus B_{M_N}(y,\epsilon_0,\frac{t_0}{2})\subseteq B_{M_N}(x \oplus y, \epsilon,t_0) \subseteq V.$$
Hence we prove that the operator $\oplus$ is continuous.

Secondly, we shall prove that the operator $\ominus$: $G\rightarrow G$ is continuous. Take any $x\in G$ and any open neighborhood $V$ of $\ominus x$.
Then there are $\epsilon\in (0,1)$ and $t_0>0$ such that $B_{M_N}(\ominus x, \epsilon,t_0)=\{z\in G:M_N(\ominus x,z,t_0)>1-\epsilon \}\subseteq V$.
Clearly, $B_{M_N}(x, \epsilon,t_0)$ is an open neighborhood of $x$.

Take any $y \in B_{M_N}(x, \epsilon,t_0)$. Then

\begin{align*}
M_N(\ominus x,\ominus y, t_0)&= M_N(e,\ominus y\oplus x, t_0)
\\&=M_N( y, x, t_0)
\\&=M_N(x, y, t_0)>1-\epsilon
\end{align*}
Hence $\ominus y\in B_{M_N}(\ominus x, \epsilon,t_0)$, this implies that $\ominus B_{M_N}(x, \epsilon,t_0)\subseteq B_{M_N}(\ominus x, \epsilon,t_0)$. Thus we have proved that the operator $\ominus$ is continuous.
\end{proof}

\section{Completions of invariant standard fuzzy metrics on gyrogroups}\label{Sec:Com}
In this section, we shall study the fuzzy metric completion of an invariant metric $d$ on a gyrogroup $G$ by showing that the gyrogroup operator $\oplus$ can be extended to the fuzzy metric completion $(\widehat{G},\widehat{M_d},\ast)$ of the standard fuzzy metric space $(G,M_d,\ast)$ of the metric space $(G,d)$ such that $(\widehat{G},\widehat{M_d},\ast)$ become a gyrogroup containing $G$ as a subgyrogroup. Let us begin with some definitions and terms.

A sequence $(x_n)_{n\in\mathbb{N}}$ in a fuzzy metric space $(X, M,\ast)$ is said to be a {\it Cauchy sequence} provided that for each $\varepsilon\in(0, 1)$ and $t>0$, there exists $n_0\in \mathbb{N}$ such that $ M(x_n, x_m, t) >1 -\varepsilon$ whenever $n, m \geq n_0$. A fuzzy metric space $(X, M,\ast)$ where every Cauchy sequence converges is called {\it complete}.

Let $(X, M,\ast)$ be a fuzzy metric space. Then a fuzzy metric completion of $(X, M,\ast)$  is a complete fuzzy metric space $(\widehat{X}, \widehat{M},\ast)$ such that $(X, M,\ast)$ is isometric to a dense subspace of $\widehat{X}$. Unfortunately, there is a fuzzy metric that does not admit any fuzzy metric completion \cite[Example 2]{GrSR}.

Let $(X,d)$ be a metric space. Denote by $a \wedge b= \text{min}\{a,b\}$ for all $a, b \in[0, 1]$, and let $M_d$ the fuzzy set defined on $X\times X\times(0,+\infty)$ by $$M(x,y,t)=\frac{t}{t+d(x,y)}.$$

Then $(X, M_d,\wedge)$ is a fuzzy metric space (see \cite[ Example 2.9 and Remark 2.10]{GeVe}). Since for any continuous $t$-norm $\ast$, one can easily show that $\ast\leq \wedge$, $(X, M_d,\ast)$ is also a fuzzy metric space for any continuous $t$-norm $\ast$. We call this fuzzy metric $(M_d,\ast)$  induced by a metric $d$ the {\it standard  fuzzy  metric} with respect to the continuous $t$-norm $\ast$. Gregori and Romaguera proved that for every metric space $(X,d)$ the standard fuzzy metric space $(X,M_d,\cdot)$ with respect to the usual multiplication $\cdot$ on $[0,1]$ admits an (up to isometry) unique  fuzzy metric completion, which is exactly the standard fuzzy metric space of the completion of $(X, d)$ with respect to the continuous $t$-norm $\cdot$ \cite[Proposition 1]{GrSR}. In fact, borrowing their skills we have the more general result:

\begin{proposition}\label{PRO}
Let $(X,d)$ be a metric space and $\ast$ a continuous $t$-norm. Then the standard fuzzy metric space $(X,M_d,\ast)$ with respect to $\ast$ admits an (up to isometry) unique fuzzy metric completion, which is exactly the standard fuzzy metric space of the completion of $(X, d)$ with respect to the $\ast$.
\end{proposition}

\begin{proof}
For any continuous $t$-norm $\star$, it is known that $(X,M_d,\star)$ is a fuzzy metric space above. Also we have the following facts:
\begin{enumerate}
\item the topologies $\tau_{M_d}$ and $\tau_d$ introduced by $(M_d,\star)$ and $d$ on $X$, respectively, are same;
\item if $(X,d)$ is a complete metric space, then $(X,M_d,\star)$ is a complete fuzzy metric space.
 \end{enumerate}

 In fact, clearly, a sequence $(x_n)_n\subseteq X$ converges to $x$ in $(X,\tau_{M_d})$ if and only if $(x_n)_n$ converges to $x$ in $(X,\tau_{d})$.  Since the space $(X,\tau_d)$ and $(X,\tau_{M_d})$ are first-countable, we have that $\tau_{M_d}=\tau_d$. Clearly, every Cauchy sequence $(x_n)_n$ in $(X,M_d,\star)$ is also a Cauchy sequence in $(X,d)$, so by fact (1) one can obtain the fact (2).

 Let $(\widehat{X},M_{\widehat{d}},\ast)$ be the standard fuzzy  metric space of the completion $(\widehat{X},\widehat{d})$ of $(X,d)$. Then $(\widehat{X},M_{\widehat{d}},\ast)$ is a complete fuzzy metric space by the fact (2) and it is the unique fuzzy metric completion of $(X,M_d,\ast)$ (up to isometry). Indeed, since there is an isometry $f$ from $(X,d)$ onto a dense subspace of $(\widehat{X},\widehat{d})$ and by fact (1), $f(X )$ is dense in $(\widehat{X},M_{\widehat{d}},\ast)$.  Furthermore, clearly, by the equality $\widehat{d}(f(x),f(y))=d(x,y)$, we have that $M_{\widehat{d}}(f(x),f(y),t)=M_d(x,y,t)$ for all $x,y\in X$ and $t>0$. So $(\widehat{X},M_{\widehat{d}},\ast)$ is a fuzzy metric completion of $(X,M_d,\ast)$, and, by \cite[Lemma 1]{GrSR}, it is unique up to isometry.
\end{proof}

\begin{definition}
Let $G$ be a gyrogroup with a fuzzy metric $(M,\ast)$ (resp. metric $d$). We say $(M,\ast)$ (resp. $d$) is invariant on $G$ if $(M,\ast)$ (resp. $d$) is invariant under left gyrotranslations and right gyrotranslations, that is
$$M(a\oplus x,a\oplus y, t)=M(x,y, t)=M( x\oplus a, y\oplus a, t)$$
for all $x,y,a\in G$ and $t>0$ (resp. $d(a\oplus x,a\oplus y)=d(x,y)=d( x\oplus a, y\oplus a)$ for all $x,y,a\in G$).
\end{definition}

Let $G$ be a gyrogroup with a fuzzy metric $(M,\ast)$. We call $(G,M,\ast)$ is a {\it fuzzy metric gyrogroup} if $G$ with the topology introduced by $(M,\ast)$ is a topological gyrogroup.
\begin{theorem}\label{The:In}
Let $G$ be a gyrogroup with an invariant fuzzy metric $(M,\ast)$. Then $(G,M,\ast)$ is a fuzzy metric gyrogroup.
\end{theorem}

\begin{proof}
Since $(M,\ast)$ is invariant on $G$, according to Theorem \ref{THM}, $G$ is a fuzzy normed gyrogroup with the corresponding fuzzy metric $(M,\ast)$.
Hence $G$ with the topology introduced by $(M,\ast)$ is a topological gyrogroup by Theorem \ref{Them:F}.
\end{proof}

\begin{theorem}\label{The:comp}
Let $d$ be an invariant metric on a gyrogroup $G$ and $(\widehat{G},\widehat{d})$ is the metric completion of the metric space $(G,d)$. Then for any continuous $t$-norm $\ast$, the standard fuzzy metric space $(\widehat{G},M_{\widehat{d}},\ast)$ of $(\widehat{G},\widehat{d})$ is the (up to isometry) unique fuzzy metric completion of the standard fuzzy metric space $(G,M_d,\ast)$ of $(G,d)$. Furthermore, $(\widehat{G},M_{\widehat{d}},\ast)$ is a fuzzy metric gyrogroup containing $(G,M_d,\ast)$ as a dense fuzzy metric subgyrogroup and $M_{\widehat{d}}$ is invariant on $\widehat{G}$.
\end{theorem}

\begin{proof}
From Proposition \ref{PRO} it follows that $(\widehat{G},M_{\widehat{d}},\ast)$ is the (up to isometry) unique fuzzy metric completion of the standard fuzzy metric space $(G,M_d,\ast)$. Now we shall prove that $(\widehat{G},M_{\widehat{d}},\ast)$ is a fuzzy metric gyrogroup containing $(G,\oplus)$ as a dense sub gyrogroup and $M_{\widehat{d}}$ is invariant on $\widehat{G}$.

Since $M_d(x,y,t)=\frac{t}{t+d(x,y)}$ for each $x,y\in G$ and $t>0$ and $d$ is invariant on $G$, so is $(M_d,\ast)$ on $G$.

Take any points $a,b\in \widehat{G}$. Consider two sequences $(a_n)_n,(b_n)_n\subseteq G$ such that $\lim\limits_{n\to \infty}a_n=a$ and $\lim\limits_{n\to \infty}b_n=b$ in $(\widehat{G},M_{\widehat{d}},\ast)$. Then we claim that the sequence $(a_n\oplus b_n)_n$ is a Cauchy sequence in $(G,M_d,\ast)$, hence in $(\widehat{G},M_{\widehat{d}},\ast)$. For this, fix $\epsilon\in (0,1)$ and $t>0$. Since $\ast$ is a continuous $t$-norm, there is $s > 0$ such that $(1-s)\ast(1-s)>1-\epsilon$. Observing that $(a_n)_n$ and $(b_n)_n$ are Cauchy sequences in $(G,M_d,\ast)$, thus there is an $n_0\in \omega$ such that $M_d(a_i,a_j,\frac{t}{2})>1-s$ and $M(b_i,b_j,\frac{t}{2})>1-s$ whenever $i,j>n_0$. Since $M_d$ is invariant on $G$, we have that
\begin{align*}
M_d(a_i\oplus b_i,a_j\oplus b_j,t)&\geq M_d(a_i\oplus b_i,a_j\oplus b_i,\frac{t}{2})\ast M_d(a_j\oplus b_i,a_j\oplus b_j,\frac{t}{2})
\\&=M_d(a_i,a_j,\frac{t}{2})\ast M_d( b_i,b_j,\frac{t}{2})
\\&= (1-s)\ast(1-s)
\\&>1-\epsilon
\end{align*}
whenever $i,j>n_0$. Hence we have prove that the sequence $(a_n\oplus b_n)_n$ is a Cauchy sequence in $(\widehat{G},M_{\widehat{d}},\ast)$.

Now we define a binary operation $\widehat{\oplus}$ on $\widehat{G}$ as follows: given two elements $a,b\in \widehat{G}$ and two sequences $(a_n)_n,(b_n)_n\subseteq G$ such that $\lim\limits_{n\to \infty}a_n=a$ and $\lim\limits_{n\to \infty}b_n=b$, $a\widehat{\oplus} b=x$, where
$x$ is the limit of $(a_n\oplus b_n)_n$ in $(\widehat{G},M_{\widehat{d}},\ast)$.

Let us show that $\widehat{\oplus}$ is well defined. Choose two sequences $(c_n)_n,(d_n)_n\subseteq G$ such that $\lim\limits_{n\to \infty}c_n=a$ and $\lim\limits_{n\to \infty}d_n=b$ in $(\widehat{G},M_{\widehat{d}},\ast)$. Then we claim that $\lim\limits_{n\to \infty}(c_n\oplus d_n)=x$ in $(\widehat{G},M_{\widehat{d}},\ast)$.

In fact, take any $\epsilon\in(0,1)$ and $t>0$. Choose $s\in(0,1)$ such that $(1-s)\ast(1-s)\ast(1-s)>1-\epsilon$, then we have that $n_0\in \omega$ such that

\begin{align*}
&M_{\widehat{d}}(x,a_n\oplus b_n, \frac{t}{3})>1-s,
\\& M_{\widehat{d}}(a_n,c_n, \frac{t}{3})>1-s,
\\&M_{\widehat{d}}(b_n,d_n, \frac{t}{3})>1-s
\end{align*}

hold whenever $n>n_0$, since $\lim\limits_{j\to \infty}(a_j\oplus b_j)=x$, $\lim\limits_{j\to \infty}d_j=b=\lim\limits_{j\to \infty}b_j$ and $\lim\limits_{j\to \infty}a_j=a=\lim\limits_{j\to \infty}c_j$ hold in $(\widehat{G},M_{\widehat{d}},\ast)$. Since $M_d$ is invariant on $G$, we have that

\begin{align*}
M_{\widehat{d}}(x, c_n\oplus d_n,t)&\geq M_{\widehat{d}}(x, a_n\oplus b_n,\frac{t}{3})\ast M_{\widehat{d}}(a_n\oplus b_n ,c_n\oplus d_n,\frac{2t}{3})
\\&\geq M_{\widehat{d}}(x, a_n\oplus b_n,\frac{t}{3})\ast M_{\widehat{d}}(a_n\oplus b_n ,c_n\oplus b_n,\frac{t}{3})\ast M_{\widehat{d}}(c_n\oplus b_n ,c_n\oplus d_n,\frac{t}{3})
\\&=M_{\widehat{d}}(x, a_n\oplus b_n,\frac{t}{3})\ast M_{\widehat{d}}(a_n ,c_n,\frac{t}{3})\ast M_{\widehat{d}}( b_n ,d_n,\frac{t}{3})
\\&\geq (1-s)\ast(1-s)\ast(1-s)
\\&>1-\epsilon
\end{align*}
 hold whenever $n>n_0$. Hence we have proved that $\lim\limits_{n\to \infty}(c_n\oplus d_n)=x$ in $(\widehat{G},M_{\widehat{d}},\ast)$. Thus, the binary operation  $\widehat{\oplus}$ is well defined.

We shall prove that $(\ominus a_n)_n$ is a Cauchy sequence in $(\widehat{G},M_{\widehat{d}},\ast)$ for any sequence $( a_n)_n\subseteq G$ such that $\lim\limits_{n\to \infty}a_n=a$ in $(\widehat{G},M_{\widehat{d}},\ast)$. Take any $\epsilon\in (0,1)$ and $t>0$. Since $\lim\limits_{n\to \infty}a_n=a$, there is $n_0\in \omega$ such that $M_{\widehat{d}}(a_i,a_j,t) > 1-\epsilon$ holds whenever $i,j>n_0$. Note that $M_d$ is invariant in $G$, so we have that
 \begin{align*}
 M_{\widehat{d}}(\ominus a_i,\ominus a_j,t)&=M_d(\ominus a_i,\ominus a_j,t)
 \\&=M_d( a_i,a_j,t)
 \\&=M_{\widehat{d}}(a_i,a_j,t)
 \\&> 1-\epsilon
 \end{align*}
 holds whenever $i,j>n_0$.

 Now we can define an operation $\widehat{\ominus}$ on $\widehat{G}$ as follows: given an element $a\in \widehat{G}$ and a sequence $(a_n)_n\subseteq G$ such that $\lim\limits_{n\to \infty}a_n=a$ in $(\widehat{G},M_{\widehat{d}},\ast)$, $\widehat{\ominus} a=x$, where
$x$ is the limit of $(\ominus a_n)_n$ in $(\widehat{G},M_{\widehat{d}},\ast)$.

 Let us show that $\widehat{\ominus}$ is well defined. Choose any sequence $(c_n)_n\subseteq G$ such that $\lim\limits_{n\to \infty}c_n=a$ in $(\widehat{G},M_{\widehat{d}},\ast)$. Then we claim that $\lim\limits_{n\to \infty}(\ominus c_n)=x$ in $(\widehat{G},M_{\widehat{d}},\ast)$. Take any $\epsilon\in(0,1)$ and $t>0$. Then one can find $s\in (0,1)$ such that $(1-s)\ast (1-s)>1-\epsilon$. Since
  $$\lim\limits_{j\to \infty}( c_j)=a=\lim\limits_{j\to \infty}( a_j)$$
  and
  $$\lim\limits_{j\to \infty}( \ominus a_j)=x,$$
 we can find $n_0\in \omega$ such that $M_d(a_n,c_n,\frac{t}{2})=M_{\widehat{d}}(a_n,c_n,\frac{t}{2})>1-s$ and $M_{\widehat{d}}(x,\ominus a_n,\frac{t}{2})>1-s$ holds whenever $n>n_0$. Since that $M_d$ is invariant in $G$, we have that
 \begin{align*}
 M_{\widehat{d}}(x,\ominus c_n,t)&\geq M_{\widehat{d}}(x,\ominus a_n,\frac{t}{2})\ast M_{\widehat{d}}(\ominus a_n,\ominus c_n,\frac{t}{2})
 \\&=M_{\widehat{d}}(x,\ominus a_n,\frac{t}{2})\ast M_d(\ominus a_n,\ominus c_n,\frac{t}{2})
 \\&=M_{\widehat{d}}(x,\ominus a_n,\frac{t}{2})\ast M_d(a_n,c_n,\frac{t}{2})
 \\&>(1-s)\ast (1-s)
 \\&>1-\epsilon
 \end{align*}
 holds whenever $n>n_0$. Hence we have prove that $\widehat{\ominus}$ is well defined.

 Next, let us show that $(\widehat{G}, \widehat{\oplus})$ is a gyrogroup. That is, the operator $\widehat{\oplus}$ satisfies the axioms $G1$-$G4$ in Definition \ref{Def:gyr}.

 Let $e$ be the identity in $G$ and take any $a,b,c$ in $\widehat{G}$ and three sequences $(a_n)_n,(b_n)_n,(c_n)_n\subseteq G$ such that $\lim\limits_{n\to \infty}a_n=a$, $\lim\limits_{n\to \infty}b_n=b$ and $\lim\limits_{n\to \infty}c_n=c$ in $(\widehat{G},M_{\widehat{d}},\ast)$.

 ($G1$): $e\widehat{\oplus} b=b$ holds for each $b\in \widehat{G}$. In fact, since $(G,\oplus)$ is a gyrogroup, we have that
 \begin{align*}
 e \widehat{\oplus} b&=\lim\limits_{n\to \infty}(e_n\oplus b_n)
 \\&=\lim\limits_{n\to \infty} b_n
 \\&=b,
 \end{align*}
 where $e_n=e$ holds for each $n\in \omega$. Similarly, one can easily show that $b\widehat{\oplus} e=b$ holds for each $b\in \widehat{G}$.

 ($G2$): For each $a\in \widehat{G}$, clearly, $$(\widehat{\ominus}a)\widehat{\oplus}a=\lim\limits_{n\to \infty}(\ominus a_n)\widehat{\oplus} \lim\limits_{n\to \infty}a_n=\lim\limits_{n\to \infty}(\ominus a_n\oplus a_n)=e=\lim\limits_{n\to \infty}(a_n\oplus(\ominus a_n))=a\widehat{\oplus}(\widehat{\ominus}a).$$

 ($G3$): Define $\widehat{\text{gyr}}[a,b](c)=\widehat{\ominus}(a\widehat{\oplus}b)\widehat{\oplus}(a\widehat{\oplus}(b\widehat{\oplus}c))$ for each $a,b,c\in\widehat{G}$.

 Firstly, let us verify that the equality ({\it$\clubsuit$}): $a\widehat{\oplus}(b\widehat{\oplus} c)=(a\widehat{\oplus}b)\widehat{\oplus}\widehat{\text{gyr}}[a,b](c)$ holds. In fact,
 \begin{align*}
 (a\widehat{\oplus}b)\widehat{\oplus}\widehat{\text{gyr}}[a,b](c)&=(a\widehat{\oplus}b)\widehat{\oplus} (\widehat{\ominus}(a\widehat{\oplus}b)\widehat{\oplus}(a\widehat{\oplus}(b\widehat{\oplus}c)))
 \\&=(\lim\limits_{n\to\infty}a_n\widehat{\oplus}\lim\limits_{n\to\infty}b_n)\widehat{\oplus} (\widehat{\ominus}(\lim\limits_{n\to\infty}a_n\widehat{\oplus}\lim\limits_{n\to\infty}b_n)\widehat{\oplus}(\lim\limits_{n\to\infty}a_n
 \widehat{\oplus}(\lim\limits_{n\to\infty}b_n\widehat{\oplus}\lim\limits_{n\to\infty}c_n)))
 \\&=\lim\limits_{n\to\infty}((a_n\oplus b_n)\oplus (\ominus(a_n\oplus b_n)\oplus(a_n
 \oplus(b_n\oplus c_n))))
 \\&=\lim\limits_{n\to\infty} ((a_n\oplus b_n)\oplus \text{gyr}[a_n,b_n](c_n))
 \\&=\lim\limits_{n\to\infty}(a_n\oplus (b_n\oplus c_n))
 \\&=\lim\limits_{n\to\infty}a_n\widehat{\oplus} (\lim\limits_{n\to\infty}b_n\widehat{\oplus}\lim\limits_{n\to\infty} c_n)
 \\&=a\widehat{\oplus}(b\widehat{\oplus} c).
 \end{align*}

 Next, we shall prove that $\widehat{\text{gyr}}[a,b]\in \text{Aut~}(\widehat{G},\widehat{\oplus})$. Clearly, the mapping $\widehat{\text{gyr}}[a,b]$ is from $(\widehat{G},\widehat{\oplus})$ to  itself.

 Firstly, we shall prove that $\widehat{\text{gyr}}[a,b]$ is a groupoid homomorphism. In fact, take any $d\in \widehat{G}$ and $(d_n)_n\subseteq G$ such that $\lim\limits_{n\to \infty}d_n=d$. Then we have that
 \begin{align*}
 \widehat{\text{gyr}}[a,b](c\widehat{\oplus}d)&=\widehat{\ominus}(a\widehat{\oplus}b)\widehat{\oplus}(a\widehat{\oplus}(b\widehat{\oplus}(c\widehat{\oplus}d)))
 \\&=\widehat{\ominus}(\lim\limits_{n\to \infty}a_n\widehat{\oplus}\lim\limits_{n\to \infty}b_n)\widehat{\oplus}(\lim\limits_{n\to \infty}a_n\widehat{\oplus}(\lim\limits_{n\to \infty}b_n\widehat{\oplus}(\lim\limits_{n\to \infty}c_n\widehat{\oplus}\lim\limits_{n\to \infty}d_n)))
 \\&=\lim\limits_{n\to \infty}(\ominus(a_n\oplus b_n)\oplus(a_n\oplus(b_n\oplus(c_n\oplus d_n))))
 \\&=\lim\limits_{n\to \infty}(\text{gry}[a_n,b_n](c_n\oplus d_n))
 \\&=\lim\limits_{n\to \infty}(\text{gry}[a_n,b_n](c_n)\oplus \text{gry}[a_n,b_n]( d_n))
 \\&=\lim\limits_{n\to \infty}((\ominus(a_n\oplus b_n)\oplus(a_n\oplus(b_n\oplus c_n)))\oplus(\ominus(a_n\oplus b_n)\oplus(a_n\oplus(b_n\oplus d_n))))
 \\&=(\widehat{\ominus}(\lim\limits_{n\to \infty}a_n\widehat{\oplus}\lim\limits_{n\to \infty} b_n)\widehat{\oplus}(\lim\limits_{n\to \infty}a_n\widehat{\oplus}(\lim\limits_{n\to \infty}b_n\widehat{\oplus}\lim\limits_{n\to \infty} c_n)))
 \\&\widehat{\oplus}(\widehat{\ominus}(\lim\limits_{n\to \infty}a_n\widehat{\oplus} \lim\limits_{n\to \infty}b_n)\widehat{\oplus}(\lim\limits_{n\to \infty}a_n\widehat{\oplus}(\lim\limits_{n\to \infty}b_n\widehat{\oplus} \lim\limits_{n\to \infty}d_n)))
 \\&=(\widehat{\ominus}(a\widehat{\oplus} b)\widehat{\oplus}(a\widehat{\oplus}(b\widehat{\oplus} c)))
 \widehat{\oplus}(\widehat{\ominus}(a\widehat{\oplus} b)\widehat{\oplus}(a\widehat{\oplus}(b\widehat{\oplus} d)))
 \\&=\widehat{\text{gyr}}[a,b](c)~\widehat{\oplus}~\widehat{\text{gyr}}[a,b](d).
 \end{align*}

Secondly, we shall prove that the equalities
 \begin{align}
&\widehat{\text{gyr}}[\widehat{\ominus}a,a](b)=b,
\\&\widehat{\ominus}a(\widehat{\oplus}a\widehat{\oplus}b)=b,
 \end{align}
hold for each $a,b\in \widehat{G}$.

By the definitions of $\widehat{\text{gyr}}[a,b]$, $\widehat{\ominus}$ and $\widehat{\oplus}$ above, we have that

\begin{align*}
\widehat{\text{gyr}}[\widehat{\ominus}a,a](b)&=\widehat{\ominus}(\widehat{\ominus}a\widehat{\oplus}a)\widehat{\oplus}(\widehat{\ominus}a\widehat{\oplus}(a\widehat{\oplus}b))
\\&=\widehat{\ominus}(\lim\limits_{n\to \infty}(\ominus a_n)\widehat{\oplus}\lim\limits_{n\to \infty}a_n)\widehat{\oplus}(\lim\limits_{n\to \infty}(\ominus a_n)\widehat{\oplus}(\lim\limits_{n\to \infty}a_n\widehat{\oplus}\lim\limits_{n\to \infty}b_n))
\\&=\lim\limits_{n\to \infty}(\ominus(\ominus a_n\oplus a_n)\oplus(\ominus a_n\oplus(a_n\oplus b_n)))
\\&=\lim\limits_{n\to \infty}(\text{gyr}[\ominus a_n, a_n](b_n))
\\&=\lim\limits_{n\to \infty}b_n
\\&=b.
\end{align*}

By the equality ({\it$\clubsuit$}), (1) and ($G1$)-($G2$) above, we have that
\begin{align*}
\widehat{\ominus}a(\widehat{\oplus}a\widehat{\oplus}b)&=(\widehat{\ominus}a\widehat{\oplus}a)\widehat{\oplus}\widehat{\text{gyr}}[\widehat{\ominus}a,a](b)
\\&=(\widehat{\ominus}a\widehat{\oplus}a)\widehat{\oplus}b
\\&=e\widehat{\oplus} b
\\&=b.
\end{align*}


Now we can prove that $\widehat{\text{gyr}}[a,b]$ is bijective. If $\widehat{\text{gyr}}[a,b](c)=\widehat{\text{gyr}}[a,b](d)$, that is,
$$\widehat{\ominus}(a\widehat{\oplus}b)\widehat{\oplus}(a\widehat{\oplus}(b\widehat{\oplus}c))=\widehat{\ominus}(a\widehat{\oplus}b)
\widehat{\oplus}(a\widehat{\oplus}(b\widehat{\oplus}d)),$$
then by (2) above one can easily obtain that $c=d$. This implies that the mapping $\widehat{\text{gyr}}[a,b]$ is injective.

Next we claim the mapping $\widehat{\text{gyr}}[a,b]$ is onto. Take any $c\in \widehat{G}$. Then
\begin{align*}
&\widehat{\text{gyr}}[a,b](\widehat{\text{gyr}}[b,a](c))
\\&=\widehat{\text{gyr}}[a,b](\widehat{\ominus}(b\widehat{\oplus}a)\widehat{\oplus}(b\widehat{\oplus}(a\widehat{\oplus}c)))
\\&=\widehat{\ominus}(a\widehat{\oplus}b)\widehat{\oplus}(a\widehat{\oplus}(b\widehat{\oplus}(\widehat{\ominus}(b\widehat{\oplus}a)\widehat{\oplus}(b\widehat{\oplus}(a\widehat{\oplus}c)))))
\\&=\widehat{\ominus}(\lim\limits_{n\to \infty}a_n\widehat{\oplus}\lim\limits_{n\to \infty}b_n)\widehat{\oplus}(\lim\limits_{n\to \infty}a_n
\widehat{\oplus}(\lim\limits_{n\to \infty}b_n\widehat{\oplus}(\widehat{\ominus}(\lim\limits_{n\to \infty}b_n\widehat{\oplus}\lim\limits_{n\to \infty}a_n)\widehat{\oplus}(\lim\limits_{n\to \infty}b_n\widehat{\oplus}(\lim\limits_{n\to \infty}a_n\widehat{\oplus}\lim\limits_{n\to \infty}c_n)))))
\\&=\lim\limits_{n\to \infty}(\ominus(a_n\oplus b_n)\oplus(a_n
\oplus(b_n\oplus(\ominus(b_n\oplus a_n)\oplus(b_n\oplus(a_n\oplus c_n))))))
\\&=\lim\limits_{n\to \infty}(\text{gyr}[a_n,b_n](\ominus(b_n\oplus a_n)\oplus(b_n\oplus(a_n\oplus c_n)))  \quad \quad \quad \text{by Proposition \ref{Pro:gyr}~} (3)
\\&=\lim\limits_{n\to \infty}(\text{gyr}[a_n,b_n](\text{gyr}[b_n,a_n](c_n))) \quad \quad \quad \quad \text{by Proposition \ref{Pro:gyr}~} (3)
\\&=\lim\limits_{n\to \infty}c_n\quad \quad \quad \quad \text{by Proposition \ref{Pro:gyr}~} (7)
\\&=c.
\end{align*}
 Hence we have proved that the mapping $\widehat{\text{gyr}}[a,b]$ is onto, furthermore, is bijective.

 ($G4$): Fix $a$ and $b$. For any $c\in \widehat{G}$, we have that
 \begin{align*}
 \widehat{\text{gyr}}[a\widehat{\oplus}b,b](c)&=\widehat{\ominus}((a\widehat{\oplus}b)\widehat{\oplus}b)\widehat{\oplus}((a\widehat{\oplus}b)\widehat{\oplus}(b\widehat{\oplus}c))
 \\&=\widehat{\ominus}((\lim\limits_{n\to\infty}a_n\widehat{\oplus}\lim\limits_{n\to\infty}b_n)\widehat{\oplus}\lim\limits_{n\to\infty}b_n)
 \widehat{\oplus}((\lim\limits_{n\to\infty}a_n\widehat{\oplus}\lim\limits_{n\to\infty}b_n)\widehat{\oplus}(\lim\limits_{n\to\infty}b_n\widehat{\oplus}\lim\limits_{n\to\infty}c_n))
 \\&=\lim\limits_{n\to\infty}(\ominus((a_n\oplus b_n)\oplus b_n)
 \oplus((a_n\oplus b_n)\oplus(b_n\oplus c_n)))
 \\&=\lim\limits_{n\to\infty}(\text{gyr}[a_n\oplus b_n,b_n](c_n))\quad \quad \quad \quad \text{by Proposition \ref{Pro:gyr}~} (3)
 \\&=\lim\limits_{n\to\infty}(\text{gyr}[a_n,b_n](c_n))\quad \quad \quad \quad \quad\quad\text{by Definition \ref{Def:gyr}~} (G4)
 \\&=\lim\limits_{n\to\infty}(\ominus(a_n\oplus b_n)\oplus(a_n\oplus(b_n\oplus c_n)))
 \\&=\widehat{\ominus}(\lim\limits_{n\to\infty}a_n\widehat{\oplus}\lim\limits_{n\to\infty} b_n)\widehat{\oplus}(\lim\limits_{n\to\infty}a_n\widehat{\oplus}(\lim\limits_{n\to\infty}b_n\widehat{\oplus} \lim\limits_{n\to\infty}c_n))
 \\&=\widehat{\ominus}(a\widehat{\oplus} b)\widehat{\oplus}(a\widehat{\oplus}(b\widehat{\oplus} c))
 \\&=\widehat{\text{gyr}}[a,b](c).
 \end{align*}
 Hence, we have proved that $\widehat{\text{gyr}}[a\widehat{\oplus}b,b]=\widehat{\text{gyr}}[a,b]$ holds for each $a,b\in \widehat{G}$.

 Thus we have proved that $(\widehat{G},\widehat{\oplus})$ is a gyrogroup. Next, we shall prove that $(M_{\widehat{d}},\ast)$ is invariant on $\widehat{G}$. Take any $a,b,c\in \widehat{G}$ and $t>0$. Choose three sequences $(a_n)_n$, $(b_n)_n$, $(c_n)_n\subseteq G$ such that $\lim\limits_{n\to\infty}a_n=a$, $\lim\limits_{n\to\infty}b_n=b$ and $\lim\limits_{n\to\infty}c_n=c$ in $(\widehat{G},M_{\widehat{d}},\ast)$. Then we have that
 \begin{align*}
 M_{\widehat{d}}(a\widehat{\oplus} b, a\widehat{\oplus} c,t)&=\frac{t}{t+\widehat{d}(a\widehat{\oplus} b, a\widehat{\oplus}c)}
 \\&=\frac{t}{t+\widehat{d}(\lim\limits_{n\to\infty}a_n\widehat{\oplus} \lim\limits_{n\to\infty}b_n, \lim\limits_{n\to\infty}a_n\widehat{\oplus}\lim\limits_{n\to\infty}c_n)}
 \\&=\frac{t}{t+\widehat{d}(\lim\limits_{n\to\infty}(a_n\oplus b_n), \lim\limits_{n\to\infty}(a_n\oplus c_n))}
 \\&=\frac{t}{t+\lim\limits_{n\to\infty}\widehat{d}(a_n\oplus b_n, a_n\oplus c_n)}\quad \quad\quad\text{the metric~~} \widehat{d} \text{~~is continuous on~} \widehat{G}
 \\&=\frac{t}{t+\lim\limits_{n\to\infty}\widehat{d}( b_n,  c_n)}\quad \quad\quad\quad\quad\quad\quad\text{the metric~~} \widehat{d} \text{~~is invariant on~} G
 \\&=\frac{t}{t+\widehat{d}(\lim\limits_{n\to\infty}b_n,  \lim\limits_{n\to\infty}c_n)}\quad ~\quad\quad\quad\quad\text{the metric~~} \widehat{d} \text{~~is continuous on~} \widehat{G}
 \\&=\frac{t}{t+\widehat{d}(b, c)}
 \\&=M_{\widehat{d}}( b, c,t).
 \end{align*}

 Similarly, one can prove that $$ M_{\widehat{d}}(b\widehat{\oplus} a, c\widehat{\oplus} a,t)= M_{\widehat{d}}(b, c,t).$$ So $(M_{\widehat{d}},\ast)$ is invariant on $\widehat{G}$.

Then from Theorem \ref{The:In} it follows that $\widehat{G}$ with the topology introduced by $M_{\widehat{d}}$ is a topological gyrogroup. Furthermore, $(\widehat{G},M_{\widehat{d}},\ast)$ is a fuzzy metric gyrogroup containing $G$ as a dense subgyrogroup, since one can easily show that $a\widehat{\oplus}b=a\oplus b$ for each $a,b\in G$.

\end{proof}
From Theorem \ref{The:comp} it follows:
\begin{corollary}\label{CC}
Every gyrogroup $G$ with an invariant metric $d$ admits an (isometric) unique complete metric space $(\widehat{G},\widehat{d})$ of $(G,d)$ such that $\widehat{G}$ with the topology introduced by $\widehat{d}$ is a topology gyrogroup containing $G$ as a dense subgyrogroup and $\widehat{d}$ is invariant on $\widehat{G}$.
\end{corollary}

\begin{theorem}\label{The:Left}
Let $(G,M,\ast)$ be a fuzzy metric gyrogroup such that $(M,\ast)$ is invariant under the left (right) gyrotranslation. If $(G,M,\ast)$ is a complete fuzzy metric, then every compatible invariant under the left (right) gyrotranslation fuzzy metric on $G$ is complete.
\end{theorem}

\begin{proof}
Let $(N,\ast)$ be a compatible invariant under the left (right) gyrotranslation fuzzy metric on $G$. Take a Cauchy sequence $(x_n)_n$ in $(G,N,\ast)$. Then we shall prove that $(x_n)_n$ is a Cauchy sequence in $(G,M,\ast)$. In fact, take any $\varepsilon\in (0,1)$ and $t>0$. Since the topologies on $G$ introduced by $(N,\ast)$ and $(M,\ast)$,respectively, are same, there is $\varepsilon_0\in (0,1)$ and $t_0>0$ such that $B_N(e,\varepsilon_0,t_0)\subseteq B_M(e,\varepsilon,t)$, where $e$ is the identity in $G$ and $B_N(e,\varepsilon_0,t_0)=\{x\in G:N(e,x,t_0)>1-\varepsilon_0\}$, similar to $B_M(e,\varepsilon,t)$. Since $(x_n)_n$ is a Cauchy sequence in $(G,N,\ast)$, for $\varepsilon_0$ and $t_0>0$ there is $j\in \omega$ such that $N(x_i,x_k, t_0)>1-\varepsilon_0$ whenever $i,k>j$. This implies that $N(e,\ominus x_i\oplus x_k, t_0)>1-\varepsilon_0$ ($N(e, x_k\oplus(\ominus x_i), t_0)>1-\varepsilon_0$), since $N$ is invariant under the left (right) gyrotranslation. Hence $\ominus x_i\oplus x_k\in B_N(e,\varepsilon_0,t_0)\subseteq B_M(e,\varepsilon,t)$ ($x_k\oplus(\ominus x_i)\in B_N(e,\varepsilon_0,t_0)\subseteq B_M(e,\varepsilon,t)$) whenever $i,k>j$. Note that $M$ is invariant under the left (right) gyrotranslation on $G$, so $M(x_i,x_k,t)>1-\varepsilon$ whenever $i,k>j$. Thus we have proved that $(x_n)_n$ is a Cauchy sequence in $(G,M,\ast)$, so from the completion of $M$ on $G$ it follows that the sequence $(x_n)_n$ converges in $(G,M,\ast)$. Since the topologies on $G$ introduced by $(N,\ast)$ and $(M,\ast)$,respectively, are same, the sequence $(x_n)_n$ converges in $(G,N,\ast)$. Thus $N$ is a complete fuzzy metric on $G$.
\end{proof}

Applying Theorems \ref{The:comp} and \ref{The:Left} we have the following result:
\begin{corollary}
If $(G,M,\ast)$ is a fuzzy metric gyrogroup such that $(M,\ast)$ is invariant, then every invariant under the left gyrotranslation fuzzy metric on the completion $(\widehat{G},\widehat{M},\ast)$ of $(G,M,\ast)$ is complete.
\end{corollary}


\begin{thebibliography}{99}
\bibitem{AbW} T.~Abe, K.~Watanabe,
\newblock Finitely generated gyrovector subspaces and orthogonal gyrodecomposition in the M\"{o}bius gyrovector space,
\newblock J. Math. Anal. Appl. \textbf{449} (2017), 77-90.

\bibitem{AlR} C.~Alegre, S.~Romaguera,
\newblock The Hahn-Banach extension theorem for fuzzy normed spaces revisited,
\newblock Abstr. Appl. Anal. (2014), 151472.

\bibitem{Atip} W.~Atiponrat,
\newblock Topological gyrogroups: generalization of topological groups,
\newblock Topol. Appl. \textbf{224} (2017), 73-82.

\bibitem{Atip1} W.~Atiponrat, R.~Maungchang,
\newblock Complete regularity of paratopological gyrogroups,
\newblock Topol. Appl. \textbf{270} (2020), 106951.

\bibitem{BaS}T.~Bag, S.K.~Samanta,
\newblock Finite dimensional fuzzy normed linear spaces,
\newblock J. Fuzzy Math. \textbf{11} (2003), 687-705.

\bibitem{Cai}Z.~Cai, S.~Lin, W.~He,
\newblock A note on paratopological loops,
\newblock Bull. Malays. Math. Sci. Soc., DOI: 10.1007/s40840-018-0616-y.

\bibitem{ChM} S.C.~Cheng, J.N.~Mordeson,
\newblock Fuzzy linear operators and fuzzy normed linear spaces,
\newblock Bull. Calcutta Math. Soc. \textbf{86} (1994), 429-436.

\bibitem{De} Z.-K.~Deng,
\newblock Fuzzy pseudo-metric spaces,
\newblock J. Math. Anal. Appl. \textbf{86} (1982), 74-95.

\bibitem{Erc} M.A.~Erceg,
\newblock Metric spaces in fuzzy set theory,
\newblock J. Math. Anal. Appl. \textbf{69} (1979), 205-230.

\bibitem{FeF} C.~Felbin,
\newblock Finite dimensional fuzzy normed linear spaces,
\newblock Fuzzy Sets Syst. \textbf{48} (1992), 239-248.

\bibitem{Ferr} M.~Ferreira,
\newblock Harmonic analysis on the M\"{o}bius gyrogroup,
\newblock J. Fourier Anal. Appl. \textbf{21}(2) (2015), 281-317.

\bibitem{GeVe} A.~George, P.~Veeramani,
\newblock On some results in fuzzy metric spaces,
\newblock Fuzzy Sets Syst. \textbf{64} (1994), 395-399.

\bibitem{GrSR} V.~Gregori, S.~Romaguera,
\newblock On completion of fuzzy metric spaces,
\newblock Fuzzy Sets Syst. \textbf{130} (2002), 399-404.

\bibitem{GSa} V.~Gregori, A.~Sapena,
\newblock On fixed-point theorems in fuzzy metric space,
\newblock Fuzzy Sets Syst. \textbf{125} (2) (2002), 245-252.

\bibitem{GuR} J.~Guti\'{e}rrez-Garc\'{\i}a, S. Romaguera, M. Sanchis,
\newblock Standard fuzzy uniform structures based on continuous $t$-norms,
\newblock Fuzzy Sets Syst. \textbf{195} (2012), 75-89.

\bibitem{KS} O.~Kaleva, S.~Seikkala,
\newblock On fuzzy metric spaces,
\newblock Fuzzy Sets Syst. \textbf{12} (1984), 215-229.

\bibitem{AnA} J.-H.~Kim, G.A.~Anastassiou, Ch. Park,
\newblock Additive $\rho$-functional inequalities in fuzzy normed spaces,
\newblock J. Comput. Anal. Appl. \textbf{21} (2016), 1115-1126.

\bibitem{KM} I.~Kramosil, J.~Michalek,
\newblock Fuzzy metrics and statistical metric spaces,
\newblock Kybernetika \textbf{11} (1975), 326-334.

\bibitem{Lee} K.Y.~Lee,
\newblock Approximation properties in fuzzy normed spaces,
\newblock Fuzzy Sets Syst. \textbf{282} (2016), 115-130.

\bibitem{MSa} S.~Macario, M.~Sanchis,
\newblock Gromov-Hausdorff convergence of non-Archimedean fuzzy metric spaces,
\newblock Fuzzy Sets Syst. \textbf{267} (2015), 62-85.

\bibitem{RaB} G.~Rano, T.~Bag, S.K.~Samanta,
\newblock Some results on fuzzy metric spaces,
\newblock J. Fuzzy Math. \textbf{19} (4) (2011), 925-938.

\bibitem{RiR} L.A.~Ricarte, S.~Romaguera,
\newblock A domain-theoretic approach to fuzzy metric spaces,
\newblock Topol. Appl. \textbf{163} (2014), 149-159.

\bibitem{RoS} S.~Romaguera, M.~Sanchis,
\newblock On fuzzy metric groups,
\newblock Fuzzy Sets Syst. \textbf{124} (2001), 109-115.

\bibitem{Sa2} I.~S\'{a}nchez, M.~Sanchis,
\newblock Complete invariant fuzzy metrics on groups,
\newblock Fuzzy Sets Syst. (2017), http://dx.doi.org/10.2016/j.fss.2016.12.019.

\bibitem{Sa} I.~S\'{a}nchez, M.~Sanchis,
\newblock Fuzzy quasi-pseudometrics on algebraic structures,
\newblock Fuzzy Sets Syst. \textbf{330} (2018), 79-86.

\bibitem{SS} B.~Schweizer, A.~Sklar,
\newblock Statistical metric spaces,
\newblock Pacific J. Math. \textbf{10} (1960), 314-334.

 \bibitem{Sh} H.~Sherwood,
\newblock On the completion of probabilistic metric spaces,
\newblock Z. Wahrscheinlichkeitstheor. Verw. Geb. \textbf{6} (1966), 62-64.

\bibitem{Su1} T.~Suksumran,
\newblock Analytic hyperbolic geometry and Albert Einstein's Special Theory of Relativity,
\newblock World Scientific, Hackensack, NJ, 2008.

\bibitem{Su} T.~Suksumran,
\newblock Essays in mathematics and its applications: In honor of Vladimir Arnold, Th. M. Rassias and P. M. Pardalos (eds.), ch. The Algebra of Gyrogroups: Cayley's Theorem, Lagrange's Theorem, and Isomorphism Theorems,
\newblock Springer, Switzerland, (2016), 369-437.

\bibitem{Suk} T.~Suksumran,
\newblock On metric structures of normed gyrogroups,
\newblock arXiv:1810.10491v2[math.MG]3 Nov 2018.

\bibitem{Ung} A.A.~Ungar,
\newblock Analytic Hyperbolic Geometry and Albert Einstein's Special Theory of Relativity,
\newblock World Scientific, 2008.

\end{thebibliography}
\end{document}